\documentclass[journal,twoside,web]{ieeecolor}
\usepackage{lcsys}
\bstctlcite{bstctl:etal, bstctl:nodash, bstctl:simpurl}
\usepackage{amsmath,mathrsfs,amsfonts,amssymb,graphicx,epsfig}
 \usepackage{amsthm}
\usepackage{subcaption}
\usepackage{color,multirow,rotating}
\usepackage{algorithm,algpseudocode,algorithmicx}
\usepackage{cite,url,framed,bm,balance,nicematrix}
\usepackage{stmaryrd}

\newtheorem{theorem}{Theorem}
\newtheorem{corollary}[theorem]{Corollary}

\newtheorem{lemma}{Lemma}

\newtheorem{remark}{Remark}

\newtheorem{example}{Example}
\usepackage{cleveref}

\newcommand{\differential}{{\rm{d}}}
\newcommand{\Jac}{{\rm{D}}}
\renewcommand{\det}{{\mathrm{det}}}

\newcommand{\vol}{{\mathrm{vol}}}

\newcommand{\Acon}{{\bm{A}_{\rm{con}}}}
\newcommand{\Aint}{{\bm{A}_{\rm{int}}}}
\newcommand{\bcon}{{\bm{b}_{\rm{con}}}}

\usepackage{tikz}
\usetikzlibrary{calc,trees,positioning,arrows,chains,shapes.geometric,decorations.pathreplacing,decorations.pathmorphing,shapes, matrix,shapes.symbols}

\allowdisplaybreaks
\usepackage{textcomp}
\def\BibTeX{{\rm B\kern-.05em{\sc i\kern-.025em b}\kern-.08em
    T\kern-.1667em\lower.7ex\hbox{E}\kern-.125emX}}
\begin{document}
\title{Exact Computation of LTI Reach Set from\\Integrator Reach Set with Bounded Input}
\author{Shadi Haddad, Pansie Khodary, Abhishek Halder, \IEEEmembership{Senior Member, IEEE}
\thanks{Shadi Haddad is with the Department of Department of Applied Mathematics, University of California, Santa Cruz, CA 95064, USA, University of California, Santa Cruz, CA 95064, USA, (e-mail: shhaddad@ucsc.edu).}
\thanks{Pansie Khodary and Abhishek Halder are with the Department of Aerospace Engineering, Iowa State University, Ames, IA 50011, USA, (e-mail: pkhodary@iastate.edu, ahalder@iastate.edu).}}

\pagestyle{empty} 
\maketitle
\thispagestyle{empty} 

\bstctlcite{IEEEexample:BSTcontrol} 

\begin{abstract}
We present a semi-analytical method for exact computation of the boundary of the reach set of a single-input controllable linear time invariant (LTI) system with given bounds on its input range. In doing so, we deduce a parametric formula for the boundary of the reach set of an integrator linear system with time-varying bounded input. This formula generalizes recent results on the geometry of an integrator reach set with time-invariant bounded input. We show that the same ideas allow for computing the volume of the LTI reach set.
\end{abstract}





\section{Introduction}\label{sec:introduction}
Motivated by safety and performance verification of controlled dynamical systems, a vast body of works in systems-control literature have studied the problem of computing or approximating the reach sets in general, and the reach sets for linear control systems \cite{pecsvaradi1971reachable,witsenhausen1972remark,varaiya2000reach,girard2005reachability,girard2006efficient,kurzhanskiy2007ellipsoidal,kaynama2009schur} in particular. See \cite{althoff2021set} for a recent survey on this broad topic. 

For our purpose, a working definition for ``reach set" would be ``the set of states that a controlled dynamical system can reach at a fixed time with given bounds on its input range". We will make this precise in Sec. \ref{subsecMainIdea}.

In this study, we propose a new semi-analytical method to compute the reach set of a controllable single input LTI system with known bounds on its input range.

\subsubsection*{Specific Contributions}\label{subsec:contributions}
\begin{itemize}
\item We show (Sec. \ref{sec:BrunovskyNormalForm}) how certain generalizations of recent results \cite{haddad2020convex,haddad2022boundary,haddad2023curious} on the exact geometry (e.g., boundary, volume) of integrator reach sets can enable computing the controllable LTI reach sets with bounded input. Our results reveal that the controllable canonical form serves as a bridge to transfer such geometric results from the integrator to the original state coordinates. 

    \item To realize the aforementioned transfer, we derive the exact integrator reach set (Sec. \ref{sec:Boundary}) and its volume (Sec. \ref{sec:Volume}) with time-varying input range that is determined by the original LTI system. These generalizations should be of independent interest, e.g., in certifying reach set intersections in feedback linearized coordinates \cite{haddad2022certifying}.
\end{itemize}

\subsubsection*{Notations and Preliminaries}\label{subsec:prelim}
We will use the finite set notation $\llbracket n\rrbracket:=\{1,2,\hdots,n\}$. The symbols $\vol_{n},\det,\vert\cdot\vert, \langle\cdot,\cdot\rangle$ denote the $n$ dimensional Lebesgue volume, the determinant, the absolute value, and the standard Euclidean inner product, respectively. 

Consider a \emph{controllable} single input LTI system
\begin{align}
\dot{\bm{z}} = {\bm{A}}\bm{z} + {\bm{b}}v, \quad{\bm{A}}\in\mathbb{R}^{n\times n},\quad{\bm{b}}\in\mathbb{R}^{n},
\label{LTIControllableSingleInput}
\end{align}
with state $\bm{z}\in\mathbb{R}^{n}$ and input $v\in\mathbb{R}$. Let $\bm{q}^{\top}$ denote the last row of the inverse of its controllability matrix, and define the nonsingular matrix 
\begin{align}
\bm{M} := \begin{pmatrix}
\bm{q}^{\top}&
\bm{q}^{\top}{\bm{A}}&
\hdots&
\bm{q}^{\top}{\bm{A}}^{n-1}
\end{pmatrix}^{\top}.
\label{defM}
\end{align}
As is well-known \cite[Sec. 6]{antasklis2007linear}, the invertible linear map $\bm{z}\mapsto\bm{x}:=\bm{Mz}$ transforms \eqref{LTIControllableSingleInput} into a \emph{controllable canonical form}
\begin{align}
\dot{\bm{x}} = \Acon\bm{x} + \bcon v,
\label{LTIControllableCanonicalSingleInput}
\end{align}
where
\begin{align}
&\Acon := \bm{M}\bm{A}\bm{M}^{-1} = \begin{pmatrix}
0 & 1 & 0 & \hdots & 0\\
0 & 0 & 1 & \hdots & 0\\
0 & 0 & 0 & \hdots & 0\\
\vdots & \vdots & \vdots & \vdots & \vdots\\
0 & 0 & 0 & \hdots & 1\\
-c_{0} & -c_{1} & -c_{2} & \hdots & -c_{n-1}
\end{pmatrix},\nonumber\\
& \bcon :=\bm{Mb} = \left(0~0~\hdots~0~1\right)^{\top}.
\label{defAhatbhat}
\end{align}
In particular, $\Acon$ is a \emph{companion matrix} (see e.g., \cite[p. 195]{horn2012matrix}) whose last row is in terms of the (real) coefficients of the monic characteristic polynomial of ${\bm{A}}$, given by
\begin{align}
p(\lambda):= \lambda^{n} + c_{n-1}\lambda^{n-1} + \hdots + c_{1}\lambda + c_{0}.
\label{CharPolyA}
\end{align}
Letting $\bm{c} := (c_0, c_{1}, \hdots, c_{n-1})^{\top}$, we can succinctly write $\Acon$ in \eqref{defAhatbhat} as
\begin{align}
\Acon = \begin{bNiceArray}{cc}
\bm{0}_{(n-1)\times 1} & \bm{I}_{n-1}\\
&\\
\Block{1-2}{-\bm{c}^{\top}}
\end{bNiceArray}.
\label{SuccinctAhatbhat}
\end{align} 


\section{Brunovsky Normal Form and LTI Reach Set}\label{sec:BrunovskyNormalForm}
\subsection{Main Idea}\label{subsecMainIdea}
Defining a new input 
\begin{align}
u := -\langle\bm{c},\bm{x}\rangle + v
\label{defv}
\end{align}
further transforms \eqref{LTIControllableCanonicalSingleInput} into the \emph{Brunovsky normal a.k.a. the $n$th order integrator form} with state $\bm{x}$ and input $u$, given by
\begin{align}
\dot{\bm{x}}=\Aint\bm{x} +\bcon u,
\label{BrunovskyNormalFormSingleInput}
\end{align}
where $\bm{x} \in \mathbb{R}^n$ and $u \in \mathbb{R}$, and
\begin{align}
\Aint := \begin{bNiceArray}{cc}
\bm{0}_{(n-1)\times 1} & \bm{I}_{n-1}\\
&\\
\Block{1-2}{\bm{0}_{1\times n}}
\end{bNiceArray}.
\label{SuccinctAhatbhat}
\end{align}
Given initial condition $\bm{z}_{0}:=\bm{z}(t=0)$ for \eqref{LTIControllableSingleInput}, the same for \eqref{BrunovskyNormalFormSingleInput} is $\bm{x}_0 = \bm{Mz}_0$. 

For continuous bounded $u$, i.e.,  $u(\cdot)\in C\left([0,t]\right)$ and $u_{\min}\leq u(\cdot) \leq u_{\max}$, the \emph{integrator reach set} $\mathcal{X}_t\left(\{\bm{x}_{0}\}\right)\subset\mathbb{R}^{n}$ is the set of states that the controlled dynamics \eqref{BrunovskyNormalFormSingleInput} can reach at a fixed time $t$ starting from $\bm{x}_{0}\in\mathbb{R}^{n}$, i.e.,
\begin{align}
&\mathcal{X}_t\left(\{\bm{x}_0\}\right) := \{\bm{x}(t)\in\mathbb{R}^{n} \mid \eqref{BrunovskyNormalFormSingleInput}, \quad\bm{u}(s)\in[u_{\min},u_{\max}]\nonumber\\
&\qquad\qquad\qquad\qquad\qquad\qquad\quad\forall s\in[0,t]\}.
\label{DefIntReachSet}    
\end{align}
Likewise, for continuous bounded $v$, i.e.,  $v(\cdot)\in C\left([0,t]\right)$ and $v_{\min}\leq v(\cdot) \leq v_{\max}$, the \emph{LTI reach set} is 
\begin{align}
&\mathcal{Z}_t\left(\{\bm{z}_{0}\}\right):=\{\bm{z}(t)\in\mathbb{R}^{n} \mid \eqref{LTIControllableSingleInput}, \quad\bm{v}(s)\in[v_{\min},v_{\max}]\nonumber\\
&\qquad\qquad\qquad\qquad\qquad\qquad\quad\forall s\in[0,t]\}.
\label{DefLTIReachSet}    
\end{align}
Thanks to the Lyapunov convexity theorem \cite{liapounoff1940fonctions,halmos1948range,artstein1990yet,ekeland1990lyapunov}, the reach sets \eqref{DefIntReachSet} and \eqref{DefLTIReachSet} are guaranteed to be compact, convex \cite[Prop. 6.1]{yong1999stochastic}, \cite{varaiya2000reach}, and in particular, \emph{zonoids}\footnote{Zonoids are defined as the range of an atom-free vector measure. These compact convex sets can be seen as the Hausdorff limit of a sequence of zonotopes (zonotopes are affine images of the unit cube); see e.g., \cite{bolker1969class,schneider1983zonoids,goodey1993zonoids,bourgain1989approximation}.} \cite{haddad2020convex,haddad2022boundary,haddad2023curious}. 


In this work, we consider the case that $v_{\min},v_{\max}$ are constants, i.e., the LTI input in \eqref{LTIControllableSingleInput} has time-invariant bounds.

We propose to use the input correspondence \eqref{defv} for explicitly computing \eqref{DefLTIReachSet} in two steps:\\
\textbf{step 1:} explicitly compute the (boundary of the compact) integrator reach set $\mathcal{X}_{t}\left(\{\bm{Mz}_0\}\right)$ for to-be-determined input range $[u_{\min}(s),u_{\max}(s)]$ $\forall s\in[0,t]$.\\
\textbf{step 2:} compute $\mathcal{Z}_{t}\left(\{\bm{z}_0\}\right) = \bm{M}^{-1}\mathcal{X}_{t}\left(\{\bm{Mz}_0\}\right)$.

For \textbf{step 1}, notice from \eqref{defv} that even when the original control input range is time invariant, i.e., $v_{\min}, v_{\max}$ are constants, still the transformed input $u$ has a time-varying range $[u_{\min}(s),u_{\max}(s)]$ $\forall s\in[0,t]$. This is because the input correspondence \eqref{defv} occurs via time-varying $\bm{x}$. Thus, we cannot apply recent results \cite{haddad2020convex,haddad2022boundary,haddad2023curious} explicitly characterizing the integrator reach set boundary with \emph{time-invariant} input range.
 
To circumvent this difficulty, we observe that $\forall s\in[0,t]$, the \emph{linear} state dependence in \eqref{defv} allows us to write
\begin{subequations}
\begin{align}
u_{\min}(s) &= -\langle\bm{c},e^{s\Acon}\bm{Mz}_{0}\rangle + I_{\min}(s),\label{vmin}\\
u_{\max}(s) &= -\langle\bm{c},e^{s\Acon}\bm{Mz}_{0}\rangle + I_{\max}(s),\label{vmax}
\end{align}
\label{uminumax}
\end{subequations}
where
\begin{subequations}
\begin{align}
I_{\min}(s) &:= \underset{v(\cdot)\in C([0,s])}{\inf} \quad I(v)\label{defImin} \\
&\quad\text{subject to}\quad v_{\min}\leq v(\cdot) \leq v_{\max},\nonumber\\
I_{\max}(s) &:= \underset{v(\cdot)\in C([0,s])}{\sup} \quad I(v)\label{defImax} \\
&\quad\text{subject to}\quad v_{\min}\leq v(\cdot) \leq v_{\max},\nonumber
\end{align}
\label{defIminImax}
\end{subequations}
and
\begin{align}
I(v) := v(s) - \!\!\int_{0}^{s}\!\!\!f(\tau)v(\tau)\differential\tau, \: f(\tau):=\big\langle\bm{c},e^{(s-\tau)\Acon}\:\bcon\big\rangle.
\label{defIv}
\end{align}

In Sec. \ref{subsec:BrunovskyInputSet}, we semi-analytically determine $I_{\min}(\cdot), I_{\max}(\cdot)$ from \eqref{defIminImax}, and thereby $u_{\min}(\cdot),u_{\max}(\cdot)$ from \eqref{uminumax}. In Sec. \ref{sec:Boundary}, we derive parametric formula for the boundary $\partial\mathcal{X}_{t}$ of the integrator reach set \eqref{DefIntReachSet} with time-varying input range. They together complete \textbf{step 1}. 

\textbf{Step 2} amounts to a simple computation: transform $\mathcal{X}_t$ via a known linear map to obtain the desired $\mathcal{Z}_t$. This is justified because the map $\bm{x}\mapsto\bm{z}=\bm{M}^{-1}\bm{x}$ is a homeomorphism, so the image of the boundary $\partial\mathcal{X}_{t}$ is the boundary of the image $\mathcal{Z}_t$. 

\begin{remark}\label{DependenceOnb}
We note that the time-varying input range $u_{\min}(\cdot), u_{\max}(\cdot)$ in \eqref{uminumax} depends on both ${\bm{A}},{\bm{b}}$. In particular, their dependence on ${\bm{b}}$ occurs through the matrix $\bm{M}$ in \eqref{defM}, which itself depends on ${\bm{b}}$ via $\bm{q}$.
\end{remark}

\subsection{Input Set in Brunovsky Coordinates}\label{subsec:BrunovskyInputSet}
To determine the time-varying input range $[u_{\min}(\cdot),u_{\max}(\cdot)]$, we first express the $f$ in \eqref{defIv} as a \emph{finite} sum involving the eigenvalues of the state matrix ${\bm{A}}$. 

\begin{lemma}\label{Lemmaf}
Suppose that the state matrix ${\bm{A}}\in\mathbb{R}^{n\times n}$ has $n$ distinct eigenvalues $\lambda_1,\hdots,\lambda_n\in\mathbb{C}$. Then the function $f$ in \eqref{defIv} can be expressed as
\begin{align}
f(\tau) = -\displaystyle\sum_{i=1}^{n}\dfrac{\lambda_{i}^{n}}{\prod_{j\neq i}(\lambda_i - \lambda_j)}e^{\lambda_i(s-\tau)},\quad 0\leq \tau \leq s.
\label{fAsFiniteSum}
\end{align}
\end{lemma}
\begin{proof}
For convenience, let $\theta:=s-\tau$. From \eqref{defAhatbhat}, the vector $e^{\theta\Acon}\:\bcon$ equals to the last column of the matrix exponential of $\theta\Acon$.

On the other hand, $\Acon$ being a companion matrix, each consecutive row of $e^{\theta\Acon}$ must be the derivative of its previous row w.r.t. $\theta$. Therefore, it suffices to determine the top right corner (i.e., first row and last column) entry of $e^{\theta\Acon}$, which we denote as $g(\theta)$.

We have
\begin{align}
g(\theta) &= \begin{pmatrix}1 & 0 & 0 & \hdots & 0\end{pmatrix}\:e^{\theta\Acon}\:\begin{pmatrix}
0&
0&
\hdots&
0&
1
\end{pmatrix}^{\!\top}\nonumber\\
&= \!\!\displaystyle\sum_{r=0}^{\infty}\!\begin{pmatrix}1 & 0 & 0 & \hdots & 0\end{pmatrix}\dfrac{\theta^{r}\bm{A}^{r}_{\rm{con}}}{r!}\begin{pmatrix}
0&
0&
\hdots&
0&
1
\end{pmatrix}^{\!\top}\nonumber\\
&=\displaystyle\sum_{r=0}^{\infty}\dfrac{\theta^r}{r!}\displaystyle\sum_{i=1}^{n}\dfrac{\lambda_i^r}{p^{\prime}(\lambda_i)},
\label{gthetaderiv}
\end{align}
where the second line uses the Taylor series for matrix exponential, $p^{\prime}(\cdot)$ denotes the derivative of the characteristic polynomial $p(\cdot)$ in \eqref{CharPolyA}, and the third line follows from a known result due to Dan Kalman \cite[equation (9)]{kalman1982generalized} that gives a formula for the top right corner entry of the $r$th power of a companion matrix in terms of the eigenvalues. The latter result in turn comes from a connection between the powers of the companion matrix and the generalized Fibonacci sequences; see also \cite{chen1996combinatorial,taher2003matrix}.

Exchanging the order of summation in \eqref{gthetaderiv}, we get
\begin{align}
g(\theta) = \displaystyle\sum_{i=1}^{n}\dfrac{1}{p^{\prime}(\lambda_i)}\displaystyle\sum_{r=0}^{\infty}\dfrac{\left(\lambda_i\theta\right)^r}{r!} = \displaystyle\sum_{i=1}^{n}\dfrac{e^{\lambda_i\theta}}{p^{\prime}(\lambda_i)}.
\label{gtheta}
\end{align}
It follows that
\begin{align}
e^{\theta\Acon}\:\bcon = \begin{pmatrix}
g(\theta)&
g^{(1)}(\theta)&
\hdots&
g^{(n-1)}(\theta)
\end{pmatrix}^{\top}
\label{LastColOfExpMatrix}    
\end{align}
where the parenthetical superscript denotes the order of derivative w.r.t. $\theta$.

Combining \eqref{defIv}, \eqref{gtheta} and \eqref{LastColOfExpMatrix}, we obtain
\begin{align}
\big\langle \bm{c},e^{\theta\Acon}\bcon\big\rangle &= \displaystyle\sum_{i=1}^{n}\dfrac{\left(c_0 + c_1\lambda_i + \hdots + c_{n-1}\lambda_i^{n-1}\right)e^{\lambda_i\theta}}{p^{\prime}(\lambda_i)}\nonumber\\
&= \displaystyle\sum_{i=1}^{n}\dfrac{\left(p\left(\lambda_i\right) -\lambda_i^{n}\right)e^{\lambda_i\theta}}{p^{\prime}(\lambda_i)}.
\label{ftaufinal}    
\end{align}
By definition, $p(\lambda_i) = 0$. Taking the logarithmic derivative of the characteristic polynomial $p(\lambda) = \prod_{i=1}^{n}(\lambda - \lambda_i)$, we have $p^{\prime}(\lambda_i)=\prod_{\stackrel{j=1}{j\neq i}}^{n}(\lambda_i-\lambda_j)$. Substituting these back in \eqref{ftaufinal}, and recalling that $\theta=s-\tau$, the proof is complete.
\end{proof}

\begin{example}\label{ExamplefComplexConjugate}
Let $n=2$ and $\lambda_{1,2} = \rho e^{\pm \iota\phi}$, $\iota := \sqrt{-1}$, with $\rho,\phi\neq 0$. Then \eqref{fAsFiniteSum} gives
$$f(\tau)=-\frac{\rho}{\sin\phi}\sin(2\phi + (s-\tau)\rho\sin(2\phi)), \quad 0\leq \tau\leq s.$$
\end{example}

\begin{example}\label{Examplef}
Let $n=3$ and $\lambda_1=1$, $\lambda_{2,3}=\pm\iota$, $\iota := \sqrt{-1}$. Then \eqref{fAsFiniteSum} gives
$$f(\tau)=-\frac{1}{2}\left(e^{s-\tau} + \cos(s-\tau) - \sin(s-\tau)\right),\, 0\leq \tau\leq s.$$
\end{example}

Theorem \ref{Thmuminumax} presented next, shows that the zeros of the continuous function $f$ become relevant for our purpose. For a fixed $s\in(0,t]$, the $f$ in \eqref{fAsFiniteSum} may have multiple (Fig. \ref{ExampleLTIf}), single or no (see footnote\footnote{as in Example \ref{Examplef}, since $e^{\theta}+\cos\theta-\sin\theta$ has no positive roots.}) zeros in its domain $[0,s]$. At our level of generality, it is not possible to bound the number of such zeros in $[0,s]$, and we will numerically find the same.

\begin{figure}[t]
    \centering
    \includegraphics[width=0.97\linewidth]{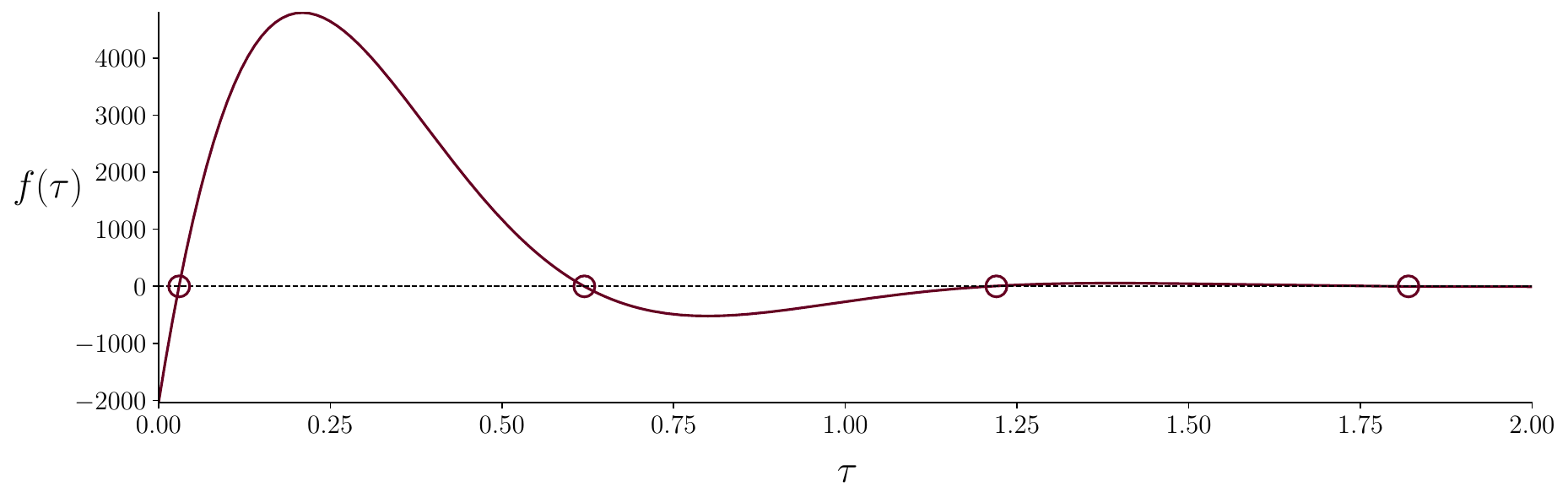}
    \caption{{ The $f(\tau)$ (\emph{solid line}) in \eqref{fAsFiniteSum} for $\bm{A}=\tiny\begin{bmatrix}
6 & 7 & 2\\
-4 & -2 & 1\\
-5 & 3 & 2
\end{bmatrix}$, and its four zeros (\emph{circular markers}) for $\tau\in[0, 2]$. Here $\bm{A}$ has one real and two complex conjugate eigenvalues.}}
\vspace*{-0.2in}
\label{ExampleLTIf}
\end{figure}

Theorem \ref{Thmuminumax} next quantifies how $u_{\min}(s),u_{\max}(s)$, and consequently the set $\mathcal{Z}_{t}\left(\{\bm{z}_0\}\right)$, are determined by the integral of $f$ over its zero sub-level and super-level sets.
\begin{theorem}\label{Thmuminumax}
Suppose $(\bm{A},\bm{b})$ is a controllable pair and matrix $\bm{A}$ has $n$ distinct eigenvalues. For any $0\leq s \leq t$, let $\mathcal{L}^{-}_{f}$ (resp. $\mathcal{L}_{f}^{++}$) denote the zero sublevel (resp. strict superlevel) set of $f$ given by \eqref{fAsFiniteSum} over $[0,s]$, i.e.,
\begin{subequations}
\begin{align}
\mathcal{L}^{-}_{f} &:= \{\tau\in[0,s] \mid f(\tau) \leq 0\}, \label{Sminus}\\ \mathcal{L}^{++}_{f} &:= \{\tau\in[0,s] \mid f(\tau) > 0\}.\label{Splusplus}
\end{align}
\label{defSminusplus}
\end{subequations}
Then the $I_{\min}(\cdot),I_{\max}(\cdot)$ in \eqref{uminumax}-\eqref{defIminImax} can be computed as
\begin{subequations}
\begin{align}
I_{\min}(s) &= v_{\min} -v_{\min}\!\int_{\tau\in\mathcal{L}_{f}^{-}}\!\!\!f(\tau)\differential\tau - v_{\max}\!\int_{\tau\in\mathcal{L}_{f}^{++}}\!\!\! f(\tau)\differential\tau,\label{Imin}\\
I_{\max}(s) &= v_{\max} -v_{\max}\!\int_{\tau\in\mathcal{L}^{-}_{f}}\!\!\!f(\tau)\differential\tau - v_{\min}\!\int_{\tau\in\mathcal{L}_{f}^{++}}\!\!\!f(\tau)\differential\tau.\label{Imax}
\end{align}
\label{IminImax}    
\end{subequations}
Furthermore, the LTI reach set \eqref{DefLTIReachSet} with time-invariant input range $[v_{\min},v_{\max}]$ can be recovered from the integrator reach set \eqref{DefIntReachSet} with time-varying input range $[u_{\min}(s),u_{\max}(s)]$ given by \eqref{uminumax}, as
$$\mathcal{Z}_{t}\left(\{\bm{z}_0\}\right) = \bm{M}^{-1}\mathcal{X}_{t}\left(\{\bm{Mz}_0\}\right).$$ 
\end{theorem}
\begin{proof}
For a fixed finite $s>0$, the function $f(\tau)$ in \eqref{defIv} is continuous and bounded in $\tau\in[0,s]$. So $v\mapsto I(v)$ is a continuous functional. Therefore, the infimum and supremum in \eqref{defIminImax} are achieved.

Because the functional $I(v)$ in \eqref{defIv} is linear, we can determine the respective optimal values in terms of the disjoint union
$$[0,s] = \mathcal{L}^{-}_{f} \cup \mathcal{L}_{f}^{++}.$$
To obtain the extrema in \eqref{defIminImax}, we substitute $v(\cdot)\equiv v_{\min}$ or $v_{\max}$ in \eqref{defIv}, depending on when $f(\tau)$ is positive or non-positive, respectively. This gives \eqref{IminImax}.

The recovery of the LTI reach set \eqref{DefLTIReachSet} from the integrator reach set \eqref{DefIntReachSet} follows from the input correspondence \eqref{defv} and the subsequent discussion in Sec. \ref{subsecMainIdea}.
\end{proof}

\section{Parametric boundary}\label{sec:Boundary}

Having determined the $I_{\min}(\cdot),I_{\max}(\cdot)$ from \eqref{IminImax}, and therefore $u_{\min}(\cdot),u_{\max}(\cdot)$ from \eqref{uminumax}, we now turn to deduce an explicit parameterization of the boundary $\partial\mathcal{X}_t\left(\{ \bm{x}_0\}\right)$ for the integrator reach set $\mathcal{X}_t\left(\{ \bm{x}_0\}\right)$ with input $u\in[u_{\min}(s),u_{\max}(s)]$, $0\leq s \leq t$. 

In contrast to prior results in \cite{haddad2020convex,haddad2022boundary,haddad2023curious} where the input uncertainty were restricted to time-invariant sets, \eqref{uminumax} requires us to consider a time-varying set $\left[u_{\min}(s),u_{\max}(s)\right]$. In this setting, Theorem \ref{Thm:ParametricBoundary} presented next, provides an explicit parameterization of $\partial\mathcal{X}_t\left(\{ \bm{x}_0\}\right)$, thus generalizing earlier developments in \cite{haddad2020convex,haddad2022boundary,haddad2023curious}.
\begin{theorem}\label{Thm:ParametricBoundary}
Consider the integrator reach set \eqref{DefIntReachSet} with time-varying range $\left[u_{\min}(s),u_{\max}(s)\right]$, $s\in[0,t]$, given by \eqref{uminumax} and \eqref{IminImax}. For $k,\ell\in\mathbb{N}$, let the indicator function $\mathbf{1}_{k\leq \ell}:=1$ for $k\leq \ell$, and $:=0$ otherwise. Define a function $\bm{\chi}(t,\bm{x}_{0}):\mathbb{R}_{>0}\times\mathbb{R}^{n}\mapsto\mathbb{R}^{n}$ component-wise: $\bm{\chi}_{k}:=\sum_{\ell = 1}^{n}\mathbf{1}_{k\leq \ell}\:\frac{t^{\ell-k}}{(\ell-k)!}\:\bm{x}_{\ell 0}$ $\forall k\in\llbracket n\rrbracket$, where $\bm{x}_{\ell 0}$ denotes the $\ell$-th component of the initial state $\bm{x}_0$.

Let the parameter vector $\bm{\sigma}=(\sigma_{1}, \sigma_{2}, \hdots, \sigma_{n-1})\in\mathcal{W}_{t}\subset \mathbb{R}^{n-1}$ where $\mathcal{W}_{t}$ is the Weyl chamber:
    \begin{align}
        \!\!\mathcal{W}_{t}:=\big\{\bm{\sigma}\in\mathbb{R}^{n-1}\!\mid 0\leq \sigma_{1} \leq \sigma_{2} \leq \hdots \leq \sigma_{n-1}\leq t \big\}.
        \label{Sspace}
    \end{align}
For $0\leq s \leq t$, let
\begin{subequations}
    \begin{align}
      \mu(s)&:=\!\left(u_{\max}(s)-u_{\min}(s)\right)\!/2,\label{defmu}\\
        \nu(s)&:=\!\left(u_{\max}(s)+u_{\min}(s)\right)\!/2,\label{defnu}
    \end{align}
\label{DefMuNu} 
\end{subequations}
and
\begin{align}
     \bm{\xi}(s):=\!\begin{pmatrix}
        \frac{s^{n-1}}{n-1} &
        \frac{s^{n-2}}{n-2} &
        \hdots &
        s &
        1
        \end{pmatrix}^{\top}.
        \label{DefXi}
    \end{align}
Then, $\bm{x}^{\textup{bdy}}
\in\partial\mathcal{X}_{t}\left(\{\bm{x}_{0}\}\right)$
    admits $\bm{\sigma}$-parameterization:
    \begin{align}           
            &\bm{x}^{\textup{bdy}}(\bm{\sigma}) \!= \bm{\chi}(t,\bm{x}_{0})\! + \!\!\int_{0}^{t} \!\!\!\nu(s)\bm{\xi}(t-s) \differential s \pm \!\!\int_{0}^{\sigma_1}\!\!\!\!\mu(s) \bm{\xi}(t-s) \differential s\nonumber\\ 
            &\!\!\!\!\!\mp \int_{\sigma_1}^{\sigma_2}\!\!\!\!\!\mu(s) \bm{\xi}(t-s)\differential s\pm \cdots\pm(-1)^{n} \!\!\!\int_{\sigma_{n-1}}^{t}\!\!\!\!\!\!\mu(s)\bm{\xi}(t-s) \differential s.       \label{ParametricBoundary}
        \end{align}   
\end{theorem}
\begin{proof}
Consider the \emph{support function} $
h_{\mathcal{X}_t\left(\{\bm{x}_{0}\}\right)}(\bm{y}) := \underset{\bm{x}\in\mathcal{X}_t\left(\{\bm{x}_{0}\}\right)}{\sup}\:\langle\bm{y},\bm{x}\rangle$
where $\bm{y}\in\mathbb{S}^{n-1}$ (unit sphere in $\mathbb{R}^{n}$). Generalizing our derivations in \cite[Thm. 1]{haddad2023curious}, we find the support function of $\mathcal{X}_t$ with time-varying input range $\left[u_{\min}(s),u_{\max}(s)\right]$ $\forall s\in[0,t]$, as
\begin{align}
\!\!&h_{\mathcal{X}_t\left(\{\bm{x}_{0}\}\right)}\left(\bm{y}\right) =  
\bigg\langle\bm{y},\bm{\chi}(t,\bm{x}_{0})+\!\!\int_{0}^{t}\!\!\nu(s)\bm{\xi}(t-s) \differential s\bigg\rangle\nonumber\\ 
&\quad\qquad+ \int_{0}^{t} \mu(s) \lvert\langle\bm{y},\bm{\xi}(t-s)\rangle\rvert\:\differential s, \quad \bm{y}\in\mathbb{S}^{n-1}.
\label{SptFnIntegratorFinal}	
\end{align}	

Next, recall that the indicator function of any convex set is a convex function that is equal \cite[Thm. 13.2]{rockafellar1997convex} to the Legendre-Fenchel conjugate (denoted as $h^{*}(\cdot)$) of its support function $h(\cdot)$. So, the boundary points $\bm{x}^{\textup{bdy}}\in\partial\mathcal{X}_t$ satisfy
\begin{align}
&h_{\mathcal{X}_t\left(\{\bm{x}_{0}\}\right)}^{*}\left(\bm{x}^{\textup{bdy}}\right) = 0, \nonumber\\
\Leftrightarrow &\underset{\bm{y}\in\mathbb{S}^{n-1}}{\min}\bigg\{\langle-\bm{x}^{\textup{bdy}}, \bm{y} \rangle+ h_{\mathcal{X}_t\left(\bm{x}_{0}\right)}(\bm{y})\bigg\} = 0.
\label{hstarequalszero}
\end{align}
Using \eqref{SptFnIntegratorFinal} and \eqref{DefXi}, we notice that the objective of the minimization problem in \eqref{hstarequalszero} contains an integral of the absolute value of the polynomial $\langle\bm{y},\bm{\xi}(.)\rangle$, which can have at most $n-1$ real roots in $[0,t]$, and therefore at most $n-1$ sign changes occur within its domain $[0,t]$.

Let us collect the real roots of this polynomial in $[0,t]$ in vector $\bm{\sigma}=(\sigma_1,\hdots,\sigma_{n-1})$ that either belongs to $\mathcal{W}_t :=\{0\leq \sigma_1\leq \sigma_2 \leq \hdots \leq \sigma_{n-1}\leq t\}\subset\mathbb{R}^{n-1}$, or are inadmissible. Then, the last integral in \eqref{SptFnIntegratorFinal} can be decomposed as a sum of $n$ sub-integrals: 
\begin{align*}
&\int_{0}^{t}|\mu(s)\langle\bm{y},\bm{\xi}(s)\rangle|\:\differential s = \pm\int_{0}^{\sigma_{1}}\!\!\!\!\!\!\!\mu(s)\langle \bm{y},\bm{\xi}(s)\rangle\:\differential s \nonumber\\
&\mp \!\!\int_{\sigma_{1}}^{\sigma_{2}}\!\!\!\!\!\!\!\mu(s)\langle \bm{y},\bm{\xi}(s)\rangle\:\differential s\pm\hdots \pm (-1)^{n-1}\!\!\!\!\int_{\sigma_{n-1}}^{t}\!\!\!\!\!\!\!\!\!\mu(s)\langle \bm{y},\bm{\xi}(s)\rangle\:\differential s.
\end{align*}
Notice that even if the number of real roots in $[0,t]$ is strictly less than $n-1$, the above expression is generic in the sense the corresponding summand integrals become zero. Notice that all elements of $\mathcal{W}_t$ are the roots for some $\bm{y}\in\mathbb{S}^{n-1}$. So for \eqref{hstarequalszero} to hold for arbitrary $\bm{\sigma}\in\mathcal{W}_t$, we need to equate the minimum value of a linear function in $\bm{y}$ to zero, which is possible if and only if the coefficient of $\bm{y}$ vanishes. This gives the parametric boundary \eqref{ParametricBoundary} where the parameter $\bm{\sigma}$ sweeps the set $\mathcal{W}_{t}$.
\end{proof}
\begin{remark}\label{Remark:RecoveryOfTimeInvariantRange}
In the special case $u_{\min},u_{\max}$, and thus $\mu,\nu$ in \eqref{DefMuNu} are constants, the parametric boundary formula \eqref{ParametricBoundary} recovers \cite[eq. (26)]{haddad2023curious} via change of variable $t-s\mapsto s$.    
\end{remark}
The plus-minus appearing in the parameterization \eqref{ParametricBoundary} has the following consequence.
\begin{corollary}\label{CorollaryUpperLower}
The integrator reach set \eqref{DefIntReachSet} with time-varying range $\left[u_{\min}(\cdot),u_{\max}(\cdot)\right]$ given by \eqref{uminumax} and \eqref{IminImax}, has two bounding surfaces $\partial \mathcal{X}_{t}^{\textup{upper}}$ and $\partial \mathcal{X}_{t}^{\textup{lower}}$, i.e.,
$$\partial\mathcal{X}_{t}\left(\{\bm{x}_{0}\}\right) := \partial\mathcal{X}_{t}^{\textup{upper}}\left(\{\bm{x}_{0}\}\right)\cup \partial\mathcal{X}_{t}^{\textup{lower}}\left(\{\bm{x}_{0}\}\right).$$
\end{corollary}
\begin{proof}
The bounding surfaces $\partial \mathcal{X}_{t}^{\textup{upper}}$ and $\partial \mathcal{X}_{t}^{\textup{lower}}$ correspond to two feasible choices of alternating signs in \eqref{ParametricBoundary}. Specifically, $\bm{x}^{\textup{upper}}(\bm{\sigma})\in\partial \mathcal{X}_{t}^{\textup{upper}}, \bm{x}^{\textup{lower}}(\bm{\sigma})\in\partial \mathcal{X}_{t}^{\textup{lower}}$ admit the following parameterization:
\begin{subequations}
\begin{align}
&\bm{x}^{\textup{upper}}(\bm{\sigma}) \!=\! \bm{\chi}(t,\bm{x}_{0}) \!+ \!\!\int_{0}^{t} \!\!\!\nu(s)\bm{\xi}(t-s) \differential s + \!\!\!\int_{0}^{\sigma_1}\!\!\!\!\!\!\mu(s) \bm{\xi}(t-s) \differential s\nonumber\\
&\!-\!\!\!\int_{\sigma_1}^{\sigma_2}\!\!\!\!\!\!\mu(s) \bm{\xi}(t-s)\differential s + \cdots +\! (-1)^{n} \!\!\!\int_{\sigma_{n-1}}^{t}\!\!\!\!\!\!\!\mu(s)\bm{\xi}(t-s) \differential s,\!
\label{xupperExplicit}\\
&\bm{x}^{\textup{lower}}(\bm{\sigma}) \!=\! \bm{\chi}(t,\bm{x}_{0}) \!+ \!\!\int_{0}^{t} \!\!\!\nu(s)\bm{\xi}(t-s) \differential s - \!\!\!\int_{0}^{\sigma_1}\!\!\!\!\!\!\!\mu(s) \bm{\xi}(t-s) \differential s\nonumber\\
&\!+\!\!\!\int_{\sigma_1}^{\sigma_2}\!\!\!\!\!\!\mu(s) \bm{\xi}(t-s)\differential s - \cdots - (-1)^{n} \!\!\!\int_{\sigma_{n-1}}^{t}\!\!\!\!\!\!\!\!\mu(s)\bm{\xi}(t-s) \differential s.\!
\label{xlowerExplicit}
\end{align} 
\label{upperlowerexplicit}
\end{subequations}
\end{proof}
\begin{remark}\label{remark:boundaryonlydependsonextremal}
From (\ref{ParametricBoundary}), the boundary $\partial\mathcal{X}_{t}\left(\{\bm{x}_{0}\}\right)$ only depends on the extremal curves $u_{\min}(\cdot), u_{\max}(\cdot)$ of the time-varying input range.
\end{remark}
\begin{remark}\label{remark:semialgebraic}
Unlike the case of time invariant input range as in \cite{haddad2022boundary,haddad2023curious}, the integrator reach set $\mathcal{X}_t$ with time-varying input range $[u_{\min}(\cdot),u_{\max}(\cdot)]$, though still a zonoid, is no longer semialgebraic in general. In fact,  formula \eqref{ParametricBoundary} shows that the zonoid $\mathcal{X}_t$, and therefore $\mathcal{Z}_{t}$, is semialgebraic iff $u_{\min}(s), u_{\max}(s)$ are polynomials $\forall 0\leq s\leq t$.
\end{remark}
The following example illustrates how the above results can be brought together to
compute the LTI reach set \eqref{DefLTIReachSet} via \textbf{step 1} and \textbf{step 2} mentioned in Sec. \ref{subsecMainIdea}.

\begin{example}\label{ExampleZt}
Consider the LTI reach set $\mathcal{Z}_t$ in \eqref{DefLTIReachSet} at $t=3$ with $n=2,\bm{z}_0=\bm{0}_{2\times 1}$, for a system of the form \eqref{LTIControllableSingleInput}:
\begin{align}
\begin{pmatrix}
\dot{z}_1 \\ \dot{z}_2 \end{pmatrix} =\underbrace{\begin{pmatrix}
0.1 & 0.2 \\ -0.3 & 0.1 \end{pmatrix}}_{\bm{A}}
\begin{pmatrix}z_1 \\ z_2 \end{pmatrix}+\underbrace{\begin{pmatrix}
1 \\ 2 \end{pmatrix}}_{\bm{b}}v,
\label{example2d}
\end{align}
where the input ($v(\cdot)$) trajectories are in $\{v(\cdot)\in C([0,t])\mid v(s)\in[-0.2,0.2]\forall s\in[0,t]\}$. So $v_{\min} = -0.2,v_{\max}=0.2$.

To compute $\mathcal{Z}_t$, we follow \textbf{step 1} and \textbf{step 2} from Sec. \ref{subsecMainIdea}. Specifically, we have
\begin{align}
\bm{M}=\begin{pmatrix}
20/11 & -10/11\\
5/11 & ~~~~3/11
\end{pmatrix},\;\bm{M}^{-1}=\begin{pmatrix}
   ~3/10 & 1\\       
   -1/2 & 2 
\end{pmatrix},
\label{MinvMExample}
\end{align}
and 
\eqref{example2d} transforms to controllable canonical form
\begin{align}\begin{pmatrix}
\dot{x}_1 \\ \dot{x}_2 \end{pmatrix} &=\underbrace{\begin{pmatrix}
0 & 1 \\ -0.07 & 0.20 \end{pmatrix}
}_{\Acon}\begin{pmatrix}x_1 \\ x_2 \end{pmatrix}+\underbrace{\begin{pmatrix}
0 \\ 1 \end{pmatrix}}_{\bcon}v,
\label{example2dcan}
\end{align}
i.e., $\bm{c} := (0.07, -0.20)^{\top}$. We find that $\bm{A}$ has eigenvalues $\lambda_{1,2}=\rho e^{\pm\iota\phi}$, $\iota:=\sqrt{-1}$, with $\rho=0.2646$, $\phi=1.1832$, and the $f$ in \eqref{fAsFiniteSum} is as in \textbf{Example \ref{ExamplefComplexConjugate}}.

With this $f$, we numerically compute $I_{\min}(s), I_{\max}(s)$ from \eqref{IminImax}, and thereby the extremal trajectories $u_{\min}(s), u_{\max}(s)$ from \eqref{uminumax} $\forall s\in[0,t=3]$. Using these $u_{\min}(s), u_{\max}(s)$ and the initial condition $\bm{x}_0=\bm{M}\bm{z}_{0}=\bm{0}_{2\times 1}$, we use \eqref{ParametricBoundary} to explicitly compute $\partial\mathcal{X}_{t}\left(\{\bm{0}_{2\times 1}\}\right)$. 
We used trapezoidal approximations with step-size $\Delta\tau=0.01$ for evaluating the integrals in \eqref{IminImax} and \eqref{ParametricBoundary}.
This completes the \textbf{step 1}. 

In \textbf{step 2}, we map $\partial\mathcal{X}_{t}\left(\{\bm{0}_{2\times 1}\}\right)$ back to $\partial\mathcal{Z}_{t}\left(\{\bm{0}_{2\times 1}\}\right)$ via the known linear map $\bm{x}^{{\rm{bdy}}}\mapsto \bm{M}^{-1}\bm{x}^{{\rm{bdy}}} \in \partial\mathcal{Z}_{t}$. 

Fig. \ref{figExampleLTI} plots the snapshots of $\mathcal{Z}_t\left(\{\bm{0}_{2\times 1}\}\right)$ computed as above, at $t= 1, 1.5, 2, 2.5$ and $3$, together with 8 sample state trajectories of \eqref{example2d} with the same zero initial state. These sample state trajectories correspond to 8 randomly sampled truncated Gaussian process input paths in $\{v(\cdot)\in C([0,t])\mid v(s)\in[-0.2,0.2]\forall s\in[0,t]\}$.
\begin{figure}[t]
    \centering    \includegraphics[width=0.85\linewidth]{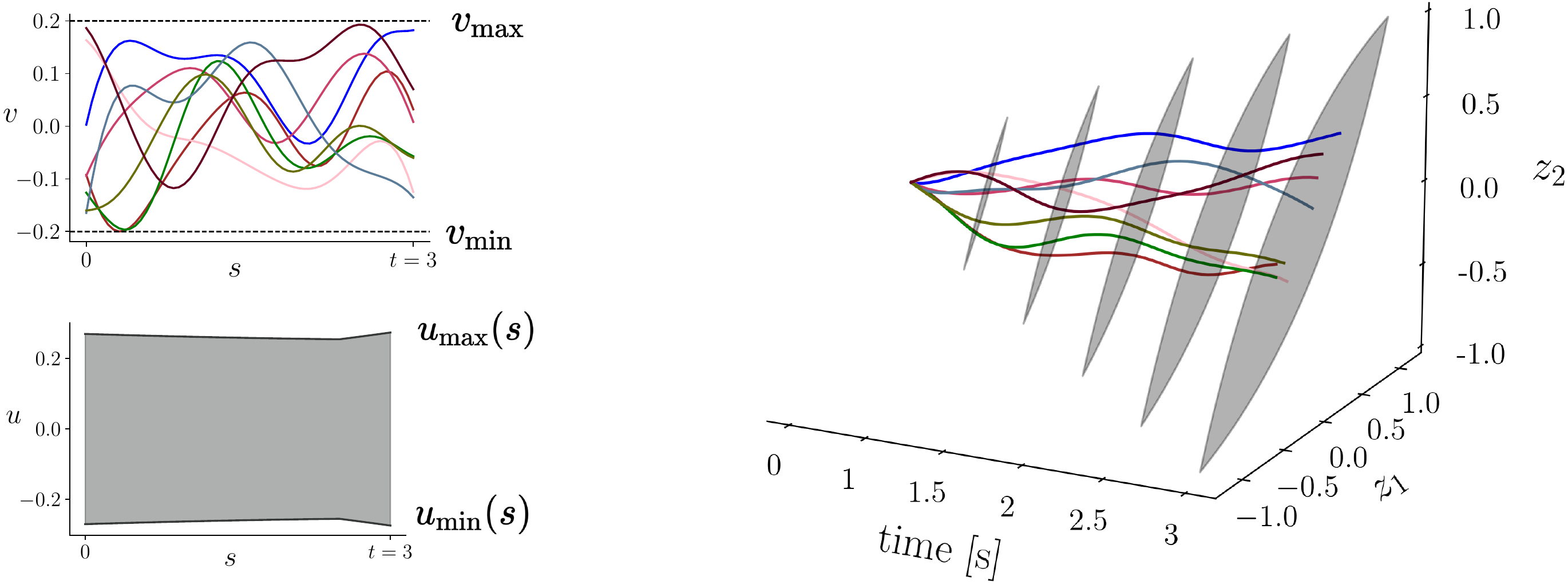}
    \caption{{\small{The reach sets $\mathcal{Z}_t\left(\{\bm{0}_{2\times 1}\}\right)$ at $t=1, 1.5, 2,2.5, 3$ (\emph{grey filled}) for \textbf{Example \ref{ExampleZt}} shown in the right plot. These sets were computed via the proposed two step method in Sec. \ref{subsecMainIdea}. The 8 sample state trajectories shown here correspond to 8 randomly sampled truncated Gaussian process input paths in $\{v(\cdot)\in C([0,t])\mid v(s)\in[-0.2,0.2]\forall s\in[0,t]\}$ shown in the top left inset plot. The bottom left inset plot shows the time-varying range $[u_{\min}(s),u_{\max}(s)]$.}}}
\vspace*{-0.2in}
\label{figExampleLTI}
\end{figure}

\end{example}

\section{Volume}\label{sec:Volume}
Building on \textbf{step 1} and \textbf{step 2} from Sec. \ref{subsecMainIdea}, we now show that the same ideas also help computing the volume of the LTI reach set, i.e., $\vol_{n}\left(\mathcal{Z}_{t}\left(\{\bm{z}_{0}\}\right)\right)$. 

As discussed in earlier works \cite{haddad2020convex}, \cite[Sec. VI]{haddad2023curious}, having a computational handle on volume is helpful in providing ground truth to quantify the conservatism of numerical algorithms which over-approximate the reach set via simpler geometric shapes such as variants of ellipsoids \cite{kurzhanski1997ellipsoidal,durieu2001multi,halder2018parameterized,halder2020smallest,haddad2021anytime} or variants of zonotopes \cite{girard2005reachability,althoff2011zonotope,althoff2015introduction,scott2016constrained,kochdumper2020sparse,kousik2022ellipsotopes,kochdumper2023open}. 



Since $\mathcal{X}_t$ is convex, each $\bm{x} \in \mathcal{X}_t$ can be written as a convex combination of two points in $\mathcal{X}_t$. In the following, we show a stronger result: any point in $\mathcal{X}_t$ can be written as a convex combination of a pair $(\bm{x}^{\textup{upper}},\bm{x}^{\textup{lower}})$ evaluated at the same parameter $\bm{\sigma}\in\mathcal{W}_{t}$. This result will find use in volume computation (Thm. \ref{Thm:Volume}).

\begin{theorem}\label{Thm:surjective}
For $\bm{\sigma}\in\mathcal{W}_t$, let $\bm{x}^{\textup{upper}}(\bm{\sigma})\in \partial \mathcal{X}_{t}^{\textup{upper}}(\bm{\sigma})$, $\bm{x}^{\textup{lower}}(\bm{\sigma})\in\partial \mathcal{X}_{t}^{\textup{lower}}(\bm{\sigma})$ where $\partial \mathcal{X}_{t}^{\textup{upper}},\partial \mathcal{X}_{t}^{\textup{lower}}$are as in Corollary \ref{CorollaryUpperLower}. For any $\bm{x}\in\mathcal{X}_{t}$, there exists $\left(\bm{\sigma},\lambda\right)\in\mathcal{W}_t \times [0,1]$ such that
\begin{align}
\bm{x}=\bm{\pi}(\bm{\sigma}, \lambda) &:=\lambda \bm{x}^{\textup{upper}}(\bm{\sigma})+(1-\lambda)\bm{x}^{\textup{lower}}(\bm{\sigma}),
\label{mapx_s}
\end{align}
i.e., the parametric map $\bm{\pi}: \mathcal{W}_t \times [0,1] \to \mathcal{X}_t$ is surjective. 
\end{theorem}
\begin{proof}
W.l.o.g., we prove our claim in a translated coordinate system with origin at $\bm{\eta}_t:=  \bm{\chi}(t,\bm{x}_{0}) \!+ \!\!\int_{0}^{t} \!\nu(s)\bm{\xi}(t-s) \differential s$. 

From \eqref{upperlowerexplicit}, in this translated coordinates, we have $\bm{x}^{\textup{upper}}_{\text{translated}} (\bm{\sigma})= \bm{x}^{\textup{upper}}(\bm{\sigma}) - \bm{\eta}_t$, $\bm{x}^{\textup{lower}}_{\text{translated}}(\bm{\sigma}) = \bm{x}^{\textup{lower}}(\bm{\sigma}) - \bm{\eta}_t$, and the bounding hypersurfaces are bijectively related via the antipodal map $\mathscr{A}: \partial \mathcal{X}_{t,\text{translated}}^{\textup{upper}} \to \partial \mathcal{X}_{t,\text{translated}}^{\textup{lower}}$ given by
$$\bm{x}^{\textup{lower}}_{\text{translated}}(\bm{\sigma})=\mathscr{A}\left(\bm{x}_{\text{translated}}^{\textup{upper}}(\bm{\sigma})\right)=-\bm{x}_{\text{translated}}^{\textup{upper}}(\bm{\sigma}).$$
So for any fixed $\bm{\sigma}\in\mathcal{W}_t$, the line segment $\bm{\ell}(\bm{\sigma}) := \lambda {\bm{x}}_{\text{translated}}^{\textup{upper}}(\bm{\sigma})+(1-\lambda)\bm{x}_{\text{translated}}^{\textup{lower}}(\bm{\sigma})$ generated by varying $\lambda \in [0, 1]$, will pass through the origin.
Since $\mathcal{X}_t$ or equivalently $\mathcal{X}_{t,\text{translated}}$ is convex and compact, for any $\tilde{\bm{x}}\in\mathcal{X}_{t,\text{translated}}$, the line through $\tilde{\bm{x}}$ and the origin crosses the boundary at the antipodal points $\bm{x}_{\text{translated}}^{\textup{upper}}(\bm{\sigma})$ and  $\bm{x}_{\text{translated}}^{\textup{lower}}(\bm{\sigma})$ for some $\bm{\sigma} \in \mathcal{W}_t$ at $\lambda=0,1$ respectively. Thus $\bm{\pi}$ is surjective.
\end{proof}
\begin{remark}\label{Remark:InteriorParam}
We clarify here that \eqref{ParametricBoundary} gives a parameterization of $\partial\mathcal{X}_t$, while \eqref{mapx_s} gives a parameterization of $\mathcal{X}_t$. 
\end{remark}
\begin{theorem}\label{Thm:Volume}
The $n$ dimensional Lebesgue volume of the  LTI reach set \eqref{DefLTIReachSet} at time $t$, is
\begin{align}
\!\!\!\vol_{n}\left(\mathcal{Z}_{t}\left(\{\bm{z}_{0}\}\right)\right)=\frac{1}{|\det(\bm{M})|}\int_{0}^{1} \!\!\!\int_{\mathcal{W}_t}\!\!\!  \begin{vmatrix}\operatorname{det}\left(\Jac\bm{\pi}\right)\end{vmatrix} \differential \bm{\sigma}\differential\lambda
\label{volumeZ}
\end{align} 
where the nonsingular $\bm{M}\in \mathbb{R}^{n\times n}$ is given by \eqref{defM}, $\Jac\bm{\pi}$ denotes the Jacobian of $\bm{\pi}$ in \eqref{mapx_s}, and $\bm{\sigma} \in \mathcal{W}_t$ as in \eqref{Sspace}. In particular,
\begin{align}
&\begin{vmatrix}\operatorname{det}\left(\Jac\bm{\pi}\right)\end{vmatrix} = \mu({\sigma}_1)\mu({\sigma}_2)  \hdots\mu({\sigma}_{n-1})|(4\lambda-2)^{n-1}|
\nonumber \\
 &
\begin{vmatrix} 
\det\begin{pmatrix}
{\dfrac{(t-\sigma_1)^{n-1}}{(n-1)!}}&\cdots&\dfrac{(t-\sigma_{{n-1}})^{n-1}}{(n-1)!}& \zeta_{1}(\bm{\sigma})\\
\vdots&\vdots&\vdots&\vdots\\
(t-\sigma_1)&\cdots&(t-\sigma_{n-1})& \zeta_{n-1}(\bm{\sigma})\\
1&\cdots&1& \zeta_{n}(\bm{\sigma})
\end{pmatrix}
\end{vmatrix},
\label{jacobian}    
\end{align}
where $\bm{\zeta}(\bm{\sigma}):=\bm{x}^{\textup{upper}}(\bm{\sigma})-\bm{x}^{\textup{lower}}(\bm{\sigma})$.
\end{theorem}
\begin{proof}
Since $\mathcal{Z}_{t}\left(\{\bm{z}_0\}\right) = \bm{M}^{-1}\mathcal{X}_{t}\left(\{\bm{Mz}_0\}\right)$, and $\vol_{n}$ is translation invariant, we have 
\begin{align}
\vol_{n}\left(\mathcal{Z}_{t}\left(\{\bm{z}_0\}\right)\right) = \frac{1}{|\det(\bm{M})|}\vol_{n}\left(\mathcal{X}_{t}\left(\{\bm{0}\}\right)\right).
\label{TranslationInvariant}
\end{align}
So it suffices to compute $\vol_{n}\left(\mathcal{X}_{t}\left(\{\bm{0}\}\right)\right)$.

Using the parameterization \eqref{mapx_s}, we get 
\begin{align}
\differential\:\vol_{n}\left(\mathcal{X}_{t}\left(\{\bm{0}\}\right)\right) &= \begin{vmatrix}\operatorname{det}\left(\Jac\bm{\pi}\right)\end{vmatrix} \differential \sigma_1\hdots\differential \sigma_{n-1} \differential \lambda.
\label{diffvolume}
\end{align}
From \eqref{mapx_s}, $\bm{\pi}$ is linear in $\lambda\in[0,1]$. Using \eqref{upperlowerexplicit} and Leibniz rule, the derivatives $\dfrac{\partial}{\partial \sigma_i}\bm{\pi}(\bm{\sigma},\lambda)$ exist $\forall i\in\llbracket n-1\rrbracket$, and are continuous in $\mathcal{W}_t$ since both $\mu(\cdot)$, $\bm{\xi}(t-\cdot)$ are continuous, and the product of continuous functions is continuous. Thus, the map $\bm{\pi}$ is $C^{1}\left(\mathcal{W}_t \times [0,1]\right)$. By Sard's theorem \cite[Ch. 2]{milnor1997topology}, the \emph{set of critical values} (image of the set of \emph{critical points} in $\mathcal{W}_t\times [0,1]$ where $\operatorname{det}\left(\Jac\bm{\pi}\right)=0$) has $n$ dimensional Lebesgue measure zero. Therefore, \eqref{TranslationInvariant}-\eqref{diffvolume} yield \eqref{volumeZ}.

Using \eqref{ParametricBoundary} and Corollary \ref{CorollaryUpperLower}, direct computation gives
\begin{align*}
&\begin{vmatrix}\operatorname{det}\left(\Jac\bm{\pi}\right)\end{vmatrix} =\begin{vmatrix}\det\left(
\dfrac{\partial \bm{\pi}(\bm{\sigma}, \lambda)}{\partial \bm{\sigma}} \right. & \left.\dfrac{\partial \bm{\pi}(\bm{\sigma}, \lambda)}{\partial \lambda}\right)
\end{vmatrix}=\eqref{jacobian}.
\end{align*}
This completes the proof.
\end{proof}


\begin{example}\label{Example2dVol}
Consider the reach set $\mathcal{Z}_t(\{\bm{0}_{2\times 1}\})$ for \eqref{example2d} at $t=2$ as in \textbf{Example \ref{ExampleZt}}. We have 
\begin{align*}
\zeta_1&=2\left(\int_{0}^{\sigma_1} \!\!\!\!(t-\tau)~ \differential \tau -\int_{\sigma_1}^{t} \!\!\!\!(t-\tau) ~\differential \tau \right)= t^2 - 2(t-{\sigma_1})^{2},\\
\zeta_2&=2\left(\int_{0}^{\sigma_1} \!\!\!\! \differential \tau -\int_{\sigma_1}^{t}  \!\!\!\!\differential \tau  \right)=4{\sigma_1} -2t.
\end{align*}
Thus \eqref{jacobian} becomes
\begin{align}
\begin{vmatrix}\operatorname{det}\left(\Jac\bm{\pi}\right)\!\end{vmatrix}\! = &\mu({\sigma}_1)|4\lambda-2| \!\begin{vmatrix}\det\!\!\begin{pmatrix}
t-\sigma_1 & t^{2}-2(t-{\sigma_1})^{2}\\
1&4{\sigma_1} -2t
\end{pmatrix}\!\!\end{vmatrix}\!\nonumber\\
=& \mu({\sigma}_1)~|4\lambda-2| ~ |-\sigma_1^2 - (t-\sigma_1)^2|.
\label{AbsDetJacExample}
\end{align}
Using \eqref{MinvMExample} and \eqref{AbsDetJacExample}, formula \eqref{volumeZ} then yields
\begin{align*}
\vol_{2}\left(\mathcal{Z}_{2}\right)\!=&~\dfrac{11}{10}\!\int_{0}^{1} \!\!\!\int_{\sigma_1}^{2}
\!\!\mu(\sigma_{1})|4\lambda-2|\!\left(\sigma_1^2 + (2-\sigma_1)^2\right)\differential \sigma_1 \differential \lambda\nonumber\\
\approx&~ 0.3043,
\end{align*} 
using the same $\mu(\cdot)$ as in \textbf{Example \ref{ExampleZt}}, and the trapezoidal method with step-size $\Delta\tau=0.01$ to estimate the integral. 
\end{example}



\section{Concluding Remarks}\label{sec:conclusions}
This work demonstrates that the boundary and volume of a controllable single input LTI reach set with \emph{time-invariant} input range can be computed exactly from the boundary and volume of an integrator reach set with a \emph{time-varying} input range induced by the LTI system. Extending this line of ideas for multi-input LTI systems will comprise our future work.


\bibliographystyle{IEEEtran}
\bibliography{References.bib}

\end{document}